\documentclass[a4paper]{article}
\usepackage{amsmath,amssymb,amsfonts,amsthm,mathrsfs,color}

\usepackage{stmaryrd}
\usepackage{graphicx,color}
\usepackage[section]{algorithm}
\usepackage[noend]{algpseudocode}

\numberwithin{equation}{section}

\makeatletter
\def\BState{\State\hskip-\ALG@thistlm}
\makeatother

\newcommand{\qi}{\mathbf i}
\newcommand{\qj}{\mathbf j}
\newcommand{\qk}{\mathbf k}

\newtheorem{theorem}{Theorem}[section]
\newtheorem{remark}[theorem]{Remark}
\newtheorem{example}[theorem]{Example}

\newtheorem{corollary}[theorem]{Corollary}
\newtheorem{proposition}[theorem]{Proposition}
\newtheorem{definition}[theorem]{Definition}

\def\V{{\rm{Vec}}}

\def\re{{\rm{Re}}}

\def\dt{{\rm{det}}}

\newcommand{\dv}{{\rm{dv}}}

\def\R{\mathbb{R}}
\def\C{\mathbb{C}}
\def\H{{\mathbb{H}}}

\def\HCO{{\mathbb{H}}_{\rm{coq}}}
\def\ll{\llbracket}
\def\rr{\rrbracket}

\def\CP{{\rm \Psi}}

\begin{document}

\title{The number of zeros of unilateral polynomials
	over coquaternions revisited}%

\author{M.~Irene Falc\~{a}o\\
CMAT and DMA\\ University of Minho, Portugal\\
mif@math.uminho.pt\\
\and Fernando Miranda\\
CMAT and DMA\\ University of Minho, Portugal\\
fmiranda@math.uminho.pt\\
\and Ricardo Severino\\
DMA\\ University of Minho, Portugal\\
ricardo@math.uminho.pt\\
\and M. Joana Soares\\
NIPE and DMA\\ University of Minho, Portugal\\
jsoares@math.uminho.pt
}

\maketitle

\begin{abstract}
The literature on quaternionic polynomials and, in particular, on methods for determining and classifying their zero-sets,
is fast developing and reveals a growing interest on this subject. In contrast, polynomials defined over the algebra of coquaternions have received very little attention from researchers.
One of the few exceptions is the very recent paper by Janovsk\'a and Opfer [{\em Electronic Transactions on Numerical Analysis}, Volume 46, pp. 55-70, 2017], where, among other results, we can find a first attempt to  prove 
that a unilateral coquaternionic polynomial of degree $n$   
has, at most,  $n(2n-1)$ zeros.  

In this paper we present a full proof 
of the referred result, using a totally different and, from our point of view, much simpler approach.
 Also, we give a complete characterization of the zero-sets of such polynomials and present a new result giving conditions which guarantee the existence of a special type  of zeros.
An algorithm to compute and classify all the zeros of a coquaternionic polynomial  is proposed and several numerical examples are carefully constructed.
\end{abstract}


\section{Introduction}
In 1941, I. Niven, in his pioneering work \cite{Niven1941}, proved that any unilateral polynomial defined over the real algebra $\H$ of quaternions 
always has a zero in $\H$, describing, simultaneously,  
a process to compute the roots of 
 any such polynomial.
The procedure proposed by Niven has two distinct parts: the first is a process to determine which similarity classes of $\H$ contain roots 
of the polynomial and the second is a computational procedure  for determining the roots lying in each class.
 It happens that the first part of Niven's scheme
is not very practical and the problem of the determination of  the roots of a quaternionic polynomial
remained dormant for quite a while. In fact, it was only in the year 2001 that Ser\^odio, Pereira and  Vit\'oria \cite{SerodioPV2001}
proposed an efficient procedure to replace the first part of Niven's method, presenting what can be considered as the first really usable algorithm for determining the zeros of quaternionic polynomials.
After this first paper, the interest in the development of root-finding methods for quaternionic polynomials has called  the attention of many researchers; see e.g.,
\cite{DeLeoDucatiLeonardi2006,Falcao2014,MirandaFalcao2014b,PogoruiShapiro2004,SerodioSiu2001}. Most of the methods available to compute the roots of a given polynomial $P$ make use  of the so-called  {\em companion polynomial} of $P$  to replace the first part of Niven's procedure, i.e. to determine which are the  classes containing the roots, differing then in the way how the roots are computed once the classes are found.
This is precisely the case of the method introduced by Ser\^odio and Siu \cite{SerodioSiu2001} and of the closely related method later proposed in \cite{JanOpf2010}.

In contrast to the case of quaternionic polynomials, the  literature on polynomials defined over the algebra  $\HCO$ of coquaternions is very scarce: see, nevertheless,  \cite{ErdogduOzdemir2015,JanOpf2013,JanOpf2014,Ozdemir2009,PogoruiRDagnino2010} and the very recent publication by Janovk\'a and Opfer \cite{JanOpf2017}.
In this last reference, the authors aim to extend to coquaternionic polynomials their method  for polynomials over $\H$   given in \cite{JanOpf2010}.  Naturally, due to the nature of the coquaternionic algebra, in particular the fact that this is not a division algebra, significant differences occur. 
In \cite{JanOpf2017} the authors 
also state an important result relating the degree of a coquaternionic polynomial with the maximum number of zeros that the polynomial  may have.\footnote{We should observe that the results contained in \cite{JanOpf2017} are obtained for polynomials defined not only over the algebra of coquaternions, but also over two other algebras: the algebra of nectarines $\H_{\rm nec}$ and the algebra of 
coquaternionic case.} 

The main purpose of this paper is to present a  complete and  simpler proof of the referred result on the maximum number of zeros of a coquaternionic polynomial and simultaneously to describe the zero-structure of such polynomials.
A new result giving conditions which guarantee the existence of a special type of zeros ---  which we call {\em linear zeros} --- is also presented and a positive answer to a question posed in \cite{JanOpf2017} is given.

The rest of the paper is organized as follows: Section~2 contains a revision of the main definitions and results on the algebra of coquaternions. Section~3 is dedicated to  unilateral coquaternionic polynomials and contains the main results of the paper; in particular,  a revised version of the root-finding algorithm proposed 
in \cite{JanOpf2017} is given.
Finally, Section~4 contains carefully chosen examples illustrating some of the conclusions contained in Section~3.

\section{Some basic results  on coquaternions}

Let $\{1, \qi, \qj,\qk\}$ be an orthonormal basis of the Euclidean vector space $\R^{4}$ with a product given
according to the multiplication rules
$$
  \qi^2=-1, \ \qj^2=\qk^2=1,\;\;\qi\qj = - \qj\qi = \qk.
$$
This non-commutative product generates the  algebra of real coquaternions, also known as split quaternions, which we will denote by  $\HCO$.
We will embed the
space $\R^4$ in $\HCO$ by identifying the element
$q=(q_0,q_1,q_2,q_3)\in \R^4$ with the element
$q = q_0 + q_1\qi + q_2\qj + q_3\qk \in \HCO$. Thus, throughout the paper, we will not distinguish an element in $\R^4$ (sometimes written as a column vector, if convenient)
from the corresponding coquaternion, unless we need to stress the context. 

The explicit multiplication of two coquaternions $p = p_0+p_1\qi+p_2\qj+p_3\qk$ and  $q = q_0+q_1\qi+q_2\qj+q_3\qk$ is given by
\begin{multline}\label{product}
 pq=  p_0q_0 - p_1 q_1 + p_2 q_2 + p_3 q_3 +(p_0 q_1 + p_1 q_0 - p_2q_3 + p_3q_2)\qi \\
 +  (p_0q_2 - p_1q_3 + p_2q_0 + p_3q_1)\qj
+ (p_0q_3 + p_1q_2 - p_2q_1 + p_3q_0)\qk. 
\end{multline}
The expression (\ref{product}) shows that the product $pq$ can be computed using matrices as 
$M_p q$, where 
\begin{equation}\label{matrixLq}
M_p=\left(
\begin{array}{cccc}
p_0 &-p_1& p_2& p_3\\
p_1 &p_0 &p_3& -p_2\\
p_2 &p_3 &p_0&-p_1\\
p_3 &-p_2 &p_1 &p_0
\end{array}\right)
\end{equation}
and where, naturally, when performing the matrix multiplication, we identify  $q$ with the column vector $(q_0, q_1, q_2,q_3)^{\rm T}.$

Given a coquaternion $q= q_0+q_1\qi+q_2\qj+q_3\qk \in \HCO$,
its {\em conjugate} $\overline{q}$ is defined as
$\overline{q}=q_0-q_1\qi-q_2\qj-q_3\qk$;
the number $q_0$ is called the {\em real part} of $q$ and denoted by $\re (q)$ and
the {\em vector part}   of $q$, denoted by $\V (q)$, is given by $\V( q)= q_1\qi+q_2\qj+q_3\qk$.

We will identify the set of coquaternions whose vector part is zero with the set $\R$ of real numbers.
We will also consider three particularly  important subspaces of dimension two of $\HCO$, usually called the {\em canonical planes} or {\em cycle planes}.  The first is $\{q \in \HCO:q=a+b\, \qi, a, b \in \R\} $ which, naturally, we  identify with the complex plane $\C$; the second, which we denote by ${\mathbb P}$ and whose elements are usually called
{\em perplex numbers} is
given  by ${\mathbb{P}}=\{ q \in \HCO:q=a+b\, \qj, a, b \in \R\}$ and  corresponds to the classical {\em Minkowski plane};
the third, denoted by ${\mathbb{D}}$, is the subspace of the so-called {\em dual numbers}, ${\mathbb D}=\{ q \in \HCO:q=a+b\, (\qi+\qj), a, b \in \R\}$ and
can be identified with the classical {\em Laguerre plane}. 
We call {\textit{determinant}} of $q$ and denote by  $\dt(q)$ the quantity
given by
\begin{equation} \label{detq}
\dt(q) =q \,\overline{q}=q_0^2+q_1^2-q_2^2-q_3^2.
\end{equation}

\begin{remark}{\rm
Other authors use different notations and denominations for the value given by (\ref{detq}); see e.g. \cite{Antonuccio2015,ErdogduOzdemir2013,JanOpf2014,JanOpf2017}.  We propose the use of the term determinant, since it is known that
every coquaternion $q=q_0+q_1 \qi+q_2\qj+q_3 q_k$ can be represented  by the real matrix 
$\begin{pmatrix}
q_0+q_3&q_1+q_2\\
q_2-q_1&q_0-q_3
\end{pmatrix}
$
whose determinant is precisely the value given by (\ref{detq}).}

\end{remark}

Contrary to what happens in the case of quaternions, not all non-zero coquaternions are invertible. It can be shown that 
a coquaternion $q$ is invertible  if and only if $\dt( q) \ne 0$. In that case, we have
$q^{-1} = \frac{\overline{q}}{\dt(q) }.$
A non-invertible element $q \in \HCO$ is also called {\textit{singular}}. It can also be shown that
a coquaternion $q$ is  singular if and only if it is a {\textit{zero divisor}}, i.e.
there exist $r,s \in \HCO, r, s \ne 0$ such that $rq=qs=0$. 

We now recall the concept of similarity in the set of coquaternions.
\begin{definition}
We say that a  coquaternion  $q$ is similar to a coquaternion $p$, and write  
$q\sim p$, if there exists an invertible coquaternion $h$ such that $q=h^{-1} p h.$
\end{definition}
This is an equivalence relation, partitioning $\HCO$ in the so-called {\em similarity-classes}, defined, for each $q \in \HCO$ by
$[q]=\{p\in \HCO: p\sim q\}$.
It can easily be shown that 
$[q]=\{q\}$ if and only if $q \in \R.$

It is a well-known result that any Hamilton quaternion is similar to a complex number with non-negative real part --- see e.g. \cite[Lemma 3]{Brenner1951} ---  thus allowing  the choice of that special form for the representative of any similarity class.  The following result shows that the situation is different in the case of coquaternions, where three different types of representatives need to be used (for non-real coquaternions).

In what follows, given a coquaternion  $q=q_0+q_1 \qi+q_2 \qj+q_3 \qk=q_0+\V( q)$, we will  use 
$\dv(q)$ to denote the determinant of the vector part of $q$, i.e. $\dv(q):=\dt(\V (q))$.

\begin{theorem}(\cite{falcaoICCSA2017},\cite{KulaYayli2007}) \label{thmJordanForms}
Let $q=q_0+q_1 \qi+q_2 \qj+q_3 \qk$ be a non-real coquaternion. 
Then:
\begin{enumerate}
\item[\rm(i)] If $\dv(q)>0$, $q$ is similar to the complex number $q_0+\sqrt{\dv(q)}\, \qi$, i.e. $[q]=[q_0+\sqrt{\dv(q)}\, \qi]$;
\item[\rm(ii)]  if $\dv(q)<0$,  $q$ is similar to the perplex number $q_0+\sqrt{-\dv(q)}\, \qj$, i.e.  $[q]=[q_0+\sqrt{-\dv(q)}\, \qj]$;
\item[\rm(iii)] if $\dv(q)=0$,  $q$ is similar to the dual number $q_0+\qi+\qj$, i.e. $[q]=[q_0+\qi+\qj] $.
\end{enumerate}
\end{theorem}
\vskip .3cm
The proof contained in \cite{KulaYayli2007} for the cases (i)-(ii) and in \cite{falcaoICCSA2017} for the case (iii)
gives explicit expressions on how we may choose $h \in \HCO$ such that $h^{-1} q h$ has the specified form. We simply recall here those expressions.
\begin{enumerate}
\item[(i)]
Take $h=(q_1+\sqrt{\dv(q)})-q_3\qj+q_2 \qk $, if $q_2^2+q_3^2\ne 0$ and $h=\qj$, if $q=q_0+q_1 \qi$ with $q_1<0$.  
\item[(ii)] If $q_1^2+q_3^2 \ne 0$, take  $h=q_1-q_3\qj+(q_2 -\sqrt{-\dv(q)})\qk $, if $q_2 \le 0$  and 
 $h=(q_2+\sqrt{-\dv(q)})+q_3 \qi +q_1 \qk$, if $q_2>0$; if $ q=q_0+q_2 \qj$ with $q_2<0$, simply take $h=\qi$.  
\item[(iii)] Take $h=(1+q_1)-q_3\qj-(1-q_2)\qk$ , if $q_1+q_2\ne 0$, and
$h=(1+q_1)\qi+(1-q_1)\qj$, otherwise.
\end{enumerate}

As an immediate consequence of the previous theorem, we have the result contained in the following corollary.

\begin{corollary}
Two non-real coquaternions $p$ and $ q$ are similar if and only if they satisfy the following conditions:
\begin{equation} \label{condSim}
\re( p)=\re( q )\quad and \quad \dv(p)=\dv(q).
\end{equation}
\end{corollary}
\begin{remark}{\rm 
Since, for any coquaternion $q$, we have $\dt(q)=(\re (q))^2+\dv(q)$ conditions (\ref{condSim}) are equivalent to  
\begin{equation}
\label{condSim2}
\re( p)=\re (q) \quad and  \quad \dt(p)=\dt(q).
\end{equation}}
\end{remark}

We should emphasize that, since any $q \in\R$ is never similar to any other coquaternion, the conditions (\ref{condSim}) (or (\ref{condSim2})) guarantee the similarity of  $p$ and $q$ only if both these coquaternions are non-real.
In connection to this, Janovsk\'a and Opfer \cite{JanOpf2014} introduced the notion of {\textit{quasi-similarity}} for any two coquaternions.
\begin{definition}
We say that two coquaternions $p$ and $q$ are 
quasi-similar, and write $p{\approx} q$, if and only if 
they satisfy conditions (\ref{condSim}).
\end{definition} 
Quasi-similarity is an equivalence relation in $\HCO$; the equivalence class of $q$ with respect to this relation is called the {\textit{quasi-similarity class}} of $q$ and will be denoted by $\ll q\rr$.\footnote{It is important to refer that
we use  different notations from the ones introduced in  \cite{JanOpf2014}, where the symbol
$\overset{q}{\sim}$ is used for the quasi-similarity relation and the quasi-similarity class of $u \in \HCO$ is denoted by $[u]_q$. Since we frequently use $q$ for a coquaternion, we found convenient to adopt  different notations.}

It is convenient to introduce the following definition.
\begin{definition}
A coquaternion $q$  is said to be of Type~1, Type~2 or Type~3, depending on whether $\dv(q)>0$, $\dv(q)<0$ or $\dv(q)=0$, respectively.\footnote{A coquaternion $q$ is usually classified as {\em time-like}, {\em space-like} or {\em light-like} according to
$\dt(q)>0$, $\dt(q)<0$ or $\dt(q)=0$, respectively.  Hence, Type~1, Type~2 and Type~3 coquaternions can also be described as coquaternions whose vector part is time-like, space-like or light-like, respectively.} 
\end{definition}

We have
\begin{align}\label{similarityClass}
\llbracket q \rrbracket&= \{p \in \HCO : \re( p) = \re( q )\ and \ \dv(p)= \dv(q)\}	\nonumber \\
&=\{p_0+p_1\qi+p_2 \qj+p_3 \qk: p_0=q_0 \ and\ p_1^2-p_2^2-p_3^2= \dv(q)\}.
\end{align}
Thus, the quasi-similarity class $\ll q\rr$ can be identified with an hyperboloid in the hyperplane $\{(x_0,x_1,x_2,x_3)\in\R^4:x_0=q_0\}$. This will be:
\begin{enumerate}
\item an hyperboloid of two sheets, 
if $\dv(q) >0$, i.e. if $q$ is of Type 1; in this case $\ll q\rr=[q]=[q_0+\sqrt{\dv(q)}\,\qi]$;
\item  
an hyperboloid of one sheet, 
if $\dv(q) <0$, i.e. if $q$ is of Type 2; in this case $\ll q\rr=[q]=[q_0+\sqrt{-\dv(q)}\, \qj]$;
\item
 a degenerate hyperboloid (i.e. a cone),  
if $\dv(q)=0$, i.e. if $q$ is of Type 3; in this case, $ \ll q \rr=\ll q_0 \rr $ and:
\begin{enumerate}
\item[\rm (i)]
if $q \in \R$, $[q]=\{q_0\};$
\item[\rm (ii)]
if $q \not \in \R$, $[q]=[q_0+\qi+\qj]=\ll q_0\rr \setminus \{q_0\}$.
\end{enumerate}
\end{enumerate}

Note that no quasi-similarity class reduces to a single point and also that any quasi-similarity class contains 
a non-real element.

\section{Unilateral coquaternionic polynomials}\label{section3}
In this section we study polynomials defined over the algebra of coquaternions with special interest on the number and nature of their zeros.
\subsection{Definition and basic results}
Unlike the real or complex case, there are
several possible ways to define coquaternionic polynomials, since the coefficients can be taken to be
on the right, on the left or on both sides of the variable. In this paper, we will restrict our attention to polynomials
whose coefficients are  located on the left of the variable, i.e. we only consider the set of polynomials of the form
\begin{equation}
\label{oneSidedPols}
P(x)=c_n x^n +c_{n-1} x^{n-1} + \cdots+c_{1}x+c_0, \ c_i \in \HCO.
\end{equation}
 We define the 
 addition and multiplication of such polynomials as in the commutative case where the variable  commutes  with the coefficients.

With these operations, this set becomes   a ring, referred to as the ring of 
(left) {\em unilateral}, {\em one-sided} or {\em simple} polynomials in $\HCO$ and denoted by $\HCO[x].$\footnote{Right unilateral polynomials are defined in an analogous manner, by considering the coefficients on the right of the variable; all the results for left unilateral polynomials have corresponding results for right unilateral  polynomials and hence we restrict our study to polynomials of the first type.} 

As usual, if $c_n \ne 0$, we say that the {\em degree} of the polynomial $P(x)$ is $n$ and  refer to $c_n$ as the leading coefficient of the polynomial. 
When $c_n=1$, we say that $P(x)$ is {\em monic}.

Naturally, due to the non-commutativity of the product in $\HCO$, the product of polynomials is also non-commutative. However, as for the product of coquaternions, a polynomial with real
coefficients commutes with any other polynomial.
 
For a given coquaternion $q$, let ${\mathscr E}_q: \HCO[x] \rightarrow \HCO$ be the {\em evaluation map}
at $q$,  defined, for the  polynomial  given by (\ref{oneSidedPols}), by 
${\mathscr E}_q(P(x))=c_n q^n +c_{n-1} q^{n-1} + \cdots+c_1q+c_0.$
Due to the way we defined the product of polynomials, this map 
is not a ring homomorphism, i.e., in general, we do not have 
$
{\mathscr E}_q(P(x) Q(x)) = {\mathscr E}_q(P(x)) {\mathscr E}_q(Q(x)).
$

\begin{remark}{\rm 
Since all the polynomials considered will be in the indeterminate $x$, we will usually omit the reference to this variable and write simply
$P$ when referring to an element  $P(x) \in \HCO[x]$; an expression of the form $P(q)$, with $q\ne x$, will be used 
to denote the evaluation of $P$ at a specific value $q \in \HCO$, i.e. $P(q)={\mathscr E}_q(P(x)).$}
\end{remark}

We say that a polynomial $R\in \HCO[x]$ is a {\em right divisor} ({\em left divisor}) of the polynomial $P$, if there exists a polynomial $Q$ such that $P = QR$ ($P=RQ$).
We say that a polynomial $D$ is a divisor of the polynomial $P$, if $D$ is simultaneously a left and a right divisor of $P$.

A coquaternion $q$ such that $P(q)=0$ is called a {\em zero} or a {\em root} of $P$. We will use $Z(P)$ to denote the {\em zero-set} of $P$ i.e.\! the set of all zeros of $P$.
 
\begin{theorem}[Factor Theorem] (\cite[Proposition~(16.2)]{Lam1991}) \label{FT}
Let $P(x)$ be a given (non-zero polynomial) in $\HCO[x]$. An element $q\in \HCO$ is a zero  of $P$ if and only if $(x-q)$ is a right divisor of $P(x)$ in $\HCO[x]$, i.e.\! if and only if  there exists a polynomial $Q(x) \in \HCO[x]$ such that
$
P(x)=Q(x)(x-q).
$
\end{theorem}

This theorem has the following immediate corollary.

\begin{corollary}
Let 
$P(x)= L(x) R(x) $ with $L(x), R(x) \in \HCO[x]$. Then, 
all the zeros of $R(x)$ are zeros of $P(x)$.
\end{corollary}

\subsection{Characteristic polynomial of a class}

Given a coquaternion $q \in \HCO$,  consider the following 
 polynomial
$$
(x-q)(x-\overline{q}) =x^2 - 2 \re( q )\, x + \dt(q).
$$
Since this polynomial depends only on $ \re( q)$ and $\dt(q)$, we immediately  conclude that this is an invariant of the quasi-similarity of $q$.
We will call it the characteristic polynomial of  $\ll q\rr$ and  will denote it by 
$\CP_{\ll q\rr}$\footnote{This polynomial is more commonly referred as  the characteristic polynomial of the coquaternion $q$. We prefer to use our  denomination to emphasize the biunivocal relation between quasi-similarity classes and characteristic polynomials.}, i.e. 
\begin{equation} \label{charPol}
\CP_{\ll q\rr}(x):=(x-q)(x-\overline{q}) =x^2 - 2 \re(q) \, x + \dt(q).
\end{equation}
Note that 
the discriminant $\Delta$ of the characteristic  polynomial (\ref{charPol}) is given 
by 
$$
\Delta= 4 (\re(q))^2 -4 \dt(q)= -4 \dv(q).
$$
This means that  $\CP_{ \ll q\rr }$ will be:
\begin{enumerate}
\item[\rm (i)] 
 an irreducible polynomial (over the reals), if $\dv(q)>0$, i.e. if $q$ is of Type~1;
 \item[\rm (ii)] a polynomial
of the form $(x-r_1)(x-r_2)$ with $r_1, r_2 \in \R, r_1 \ne r_2$, if $\dv(q)<0$, i.e. if
$q$ is of Type~2;
\item[\rm (iii)]  a polynomial of the form
$(x-r)^2,$ with $ r \in \R$, if $\dv(q)=0$, i.e. if $q$ is of Type~3.
\end{enumerate}
On the other hand, any second degree monic polynomial with real coefficients is the characteristic polynomial of a (uniquely defined) quasi-similarity class.
In fact, let $p_2(x)=x^2+bx +c$ with $b,c \in \R$ and let $\Delta=b^2 -4 c$. Considering $p_2$ as a polynomial
in $\C[x]$, we have:
\begin{enumerate}
\item[\rm (i)] 
if $\Delta<0$, then $p_2$ has two (distinct) complex conjugate roots, $w$ and $\overline{w}$; hence, $p_2(x)=(x-w)(x-\overline{w})$ i.e. $p_2=\CP_{\ll w \rr}$.
\item[\rm (ii)]
If $\Delta=0$, then $p_2(x)$ has a double real root $r$, i.e. 
$p_2(x)=(x-r)^2 $ and so $p_2=\CP_{\ll r \rr} $.
\item[\rm (iii)]
If $\Delta>0$, then $p_2$ has two distinct real roots $r_1, r_2$, i.e. 
$p_2(x)=(x-r_1)(x-r_2)=x^2-(r_1+r_2) x+r_1 r_2$ and  it is easy to see that 
$p_2=\CP_{\ll p\rr }$, with $ p$ the perplex number $p=\frac{r_1+r_2}{2}+\frac{r_1-r_2}{2}\qj.$
\end{enumerate} 

 \vskip .1cm
The next theorem states an important property of the zero-set  of characteristic polynomials; see e.g.,
\cite[Theorem 4]{falcaoICCSA2017} for this and other properties of $Z(\CP_{\ll q\rr}).$
\begin{theorem} \label{theoremPropCharPol}
Let $\CP_{\ll q\rr}$ be the characteristic polynomial of a given quasi-simila\-ri\-ty class $\ll q\rr$. Then
$
\ll q\rr \subseteq Z(\CP_{\ll q\rr} ).
$
\end{theorem}

\subsection{Zeros of polynomials and the companion polynomial}
Given a polynomial $P \in \HCO[x]$, the  polynomial obtained from $P$ by replacing each coefficient by its conjugate is called the {\em conjugate of $P$} and denoted by $\overline{P}$. The properties given in the following proposition are easily verified.
\begin{proposition}\label{propConjugatePolynomial}
Let $P, Q\in \HCO[x]$. Then:
\begin{enumerate}
\item[\rm (i)] $\overline{P Q}=\overline{Q} \,\,  \overline{P}$
\item[\rm (ii)] $P  \overline{P}=\overline{P}  P$ is a polynomial with real
coefficients.
\end{enumerate}
\end{proposition}

\begin{definition}
The polynomial 
\begin{equation}\label{compPolynomial}
{\cal C}_P=\overline{P}  P=P\overline{P}
\end{equation}
is called the companion polynomial of the polynomial $P$.
\end{definition}

\begin{remark}{\rm This polynomial already appears in the famous work \cite{Niven1941} of Niven, published in 1941  --- unfortunately with no name --- and has been used by many authors under different designations.
The name {\em companion polynomial} here adopted was introduced in \cite{JanOpf2010}.}
\end{remark}

In what follows, we restrict our attention to the study of monic polynomials, i.e.  we will consider only polynomials of the 
form
\begin{equation}\label{monicPolynomials}
P(x)=x^n+c_{n-1}x^{n-1} +\ldots+c_1 x+c_0, \quad c_i \in \HCO.
\end{equation}
 Note that, in what concerns the zeros of polynomials, the study of polynomials of this type is equivalent to the study of polynomials of the 
form (\ref{oneSidedPols}) with a non-singular leading coefficient $c_n$.\footnote{The fact that $P$ is monic (or with  non singular leading coefficient) guarantees that the companion polynomial is a polynomial of degree $2n$ and avoids pathological situations such as having a non-zero polynomial whose companion polynomial is the zero polynomial; see 
\cite{JanOpf2014} for such an example.}

Niven~\cite{Niven1941} proved that every non-constant unilateral polynomial with quaternionic coefficients  always has a quaternionic zero, thus establishing that the ``Fundamental Theorem of Algebra" holds for
unilateral quaternionic polynomials.\footnote{This result was later extended to more general quaternionic polynomials in \cite{Niven1944}.}
The situation, in what concerns polynomials defined over the algebra of coquaternions is, however, different. 
In fact, as observed by \"{O}zdemir \cite[Theorem 9-i.]{Ozdemir2009}, 
any equation of the  form
$x^n-q=0$, with $n$ even and $q$ a coquaternion  with negative determinant does not have a solution; other examples of coquaternionic polynomials with no roots can be found in \cite{JanOpf2014}.

The following theorem plays an important role in the  root-finding procedure that we are going to propose.
 
\begin{theorem}\label{thmFTCP}
Let $P \in \HCO[x]$. If $z\in \HCO$ is a zero of $P$, then $\CP_{\ll z\rr}$ is a divisor of the companion polynomial
of $P$.
\end{theorem}

\begin{proof}
Let $z$ be a zero of $P$; by Theorem~\ref{FT}, we know that $(x-z)$ is a right divisor of $P$, i.e. $P(x)=Q(x)  (x-z)$
for a given polynomial $Q \in \HCO[x]$. Thus
$$\begin{array}{ll}{\cal C}_P(x)&=P(x) \overline{P}(x)=
Q(x)  (x-z)  \overline{ (Q(x)  (x-z))}\\
&=Q(x) (x-z)   (x-\overline{z}) \overline{Q}(x)=Q(x)\CP_{\ll z\rr}(x)  \overline{Q}(x)\\
&= Q(x)   \overline{Q} (x)  \CP_{z}(x) ={\cal C}_Q (x)  \CP_{\ll z\rr}(x)= \CP_{\ll z\rr}(x) {\cal C}_Q (x) ,
\end{array}
$$
where we used the fact that a polynomial with real
coefficients commutes with any other polynomial. Hence,  $\CP_{\ll z\rr}$ is a divisor of ${\cal C}_P$, as we wished to prove. \qquad 
\end{proof}

The result of the previous theorem can also  be stated  as follows: if a certain quasi-similarity class 
$\ll z\rr$ of $\HCO$ contains a zero of $P$, then its characteristic polynomial $\CP_{\ll z\rr}$ divides the companion polynomial ${\cal C}_P$. This means, in particular, that there is no point in searching for zeros of $P$ in classes whose 
characteristic polynomial is  not a factor of  ${\cal C}_P$.  We are thus led to introduce the following 
definition.

\begin{definition}
A quasi-similarity class $\ll z\rr$ of $\HCO$ is called admissible (with respect to the zeros of a given polynomial $P$) if and only if the corresponding characteristic polynomial $\CP_{\ll z\rr}$ is a divisor of the companion polynomial of $P$.
\end{definition}

If $P$ is a monic polynomial of degree $n$, its companion polynomial ${\cal C}_P$ is a polynomial with real
coefficients of degree $2n$ and, as such, has $2n$ roots in $\C$.   Let these roots be  $w_1,\overline{w}_1, \ldots,w_m $, $\overline{w}_m \in \C \setminus \R$ and $r_1, r_2, \ldots, r_{s} \in \R$, where $s=2n-2m$, $0 \le m \le n$. Then, 
$$
{\cal C}_P(x)=\underbrace{(x-w_1) (x-\overline{w}_1)}_{\CP_{\ll w_1\rr}} \dots\underbrace{(x-w_1) (x-\overline{w}_m)}_{\CP_{\ll w_m\rr}}(x-r_1)  \ldots (x-r_s)
$$
and it is clear that the  characteristic polynomials (i.e. the real monic polynomials of degree two)  which divide ${\cal C}_P$  are the $m$ irreducible polynomials
\begin{equation}\label{chPcomplex}
\CP_{\ll w_i\rr }, i=1, \ldots,m,
\end{equation} 
and the  ${s \choose 2}$ polynomials $(x-r_j)(x-r_k)$; $j=1,\ldots, s-1,\,   k=j+1,\ldots,s,$ i.e. the polynomials
\begin{equation}\label{chPreal}
  \CP_{\ll p_{jk}\rr}
\quad {\rm with }\quad 
p_{jk}=\textstyle{\frac{r_j+r_k}{2}} + \textstyle{\frac{r_j-r_k}{2}}\, \qj; j=1,\ldots, s-1, \,  k=j+1,\ldots,s.
\end{equation}
Note that if $r_j$ is a multiple root, a polynomial of the form $(x-r_j)^2=\CP_{\ll r_j\rr }$ appears.
The maximum number of  such polynomials
occurs when $m=0$ and the $2n$  real roots of $ {\cal C}_P$ are all distinct, and is equal to ${2n \choose 2}=n(2n-1).$ 
Having in mind the correspondence between characteristic polynomials and quasi-similarity classes, we can then state the following result.
 \begin{theorem}\label{GM}
 If $P$ is a polynomial of degree $n$ in $\HCO[x]$, then the zeros of $P$ belong to, at most, 
 $n(2n-1)$ quasi-similarity classes. 
 \end{theorem}
 \begin{remark}{\rm $ $
 \begin{enumerate}
 \item 
 This can be seen as the analogue, in the coquaternionic setting, of a result first established by Gordon and Motzkin
 \cite{GordonMotzkin1965} for quaternionic polynomials.
 \item
 The result of the previous theorem was first stated, by using totally different arguments, in \cite[Corollary 5.3]{JanOpf2017}.
 We have, however, to refer that the proof contained in \cite{JanOpf2017} is, not only much more elaborate than the proof here presented, but also incomplete. In fact, Corollary 5.3. follows from the use of two theorems ---  Theorem 5.1 and Theorem 5.2. --- which deal, respectively, with the case of zeros belonging to a  similarity class of the form $[u+v \qi], v>0$  or to a similarity class $ [u+ v\qj], v>0$. So, the case of zeros in a class of the type $[u+\qi+\qj]$ (i.e. zeros whose vector part has zero determinant, but which are not real), is missing. 
 \end{enumerate}}
 \end{remark}
 
\subsection{Computing the zeros from the admissible classes}

We now
explain how to compute the zeros belonging to each of the admissible classes. 
The process is an adaption to the coquaternionic setting of the already referred process for quaternionic polynomials proposed by Niven in \cite{Niven1941}. Naturally, some differences occur due to the fact that $\HCO$, contrary to the algebra of quaternions, is not a division algebra. The same idea (although using a different computational procedure) was used in  \cite{JanOpf2014} and \cite{JanOpf2017}.

Naturally, for any monic polynomial of the first degree, $P(x)=x-q$, the only root of $P(x)$ is $q$, and so we will now consider only polynomials of degree $n\ge 2$.

In what follows, $\ll q\rr =\ll q_0+\V (q) \rr$ is an  admissible class of $P$, i.e.,
$\CP_{\ll q\rr}(x)$ divides ${\cal C}_P$.  
 This implies, in particular,  that all the elements in  $\ll q \rr$ are zeros of ${\cal C}_P$; see Theorems~\ref{FT} and \ref{theoremPropCharPol}.
 
We first note that, since the product of two polynomials in $\HCO[x]$ is defined in the usual manner, we can always use  the ``Euclidean Division Algorithm" to perform the division of two polynomials, provided that the leading coefficient of the divisor is non-singular. In particular, we can use it to divide  
 $P(x)$ by  the characteristic polynomial of the quasi-similarity class ${\ll q\rr}$, i.e. by the quadratic monic polynomial
 $x^2-2 \re(q) x+\dt(q)$. If we perform this division we will obtain 
\begin{equation}\label{DivChar}
P(x)=Q(x) \CP_{\ll q\rr}(x)+ A+Bx
\end{equation}
for some polynomial $Q(x)$ and values $A$ and $B$ which depend only on the coefficients $c_i$ of the given polynomial $P$ and on the values $\re(q)$ and $\dt(q)$.   
As in the case of complex or quaternionic polynomials, expanded synthetic division can  here be used  to obtain the expressions of $A$ and $B$  (see \cite{SmoktunowiczWrobel2005} and \cite{FalcaoMirandaSeverinoSoares2017} for computational details on the use of this shortcut method for dividing any polynomial by the special quadratic polynomial $\CP_{\ll q\rr}$, in the complex and quaternionic context, respectively). Concretely, we have
\begin{subequations}\label{fAB}
\begin{equation} 
A=c_0-\det(q)\,\alpha_0\qquad \text{and} \qquad  B=c_{1}+2\re(q)\,\alpha_{0}-\det(q)\,\alpha_{1},
\end{equation}
where $\alpha_j$ satisfy the following recurrence relations:
\begin{equation}
\begin{array}{ll}
\alpha_{n-1}=0,\  \alpha_{n-2}=1, \\[2ex]
\alpha_{k}= c_{k+2}+2\re(q)\,\alpha_{k+1}-\det(q)\,\alpha_{k+2};\ k=n-3,n-2,\dots,0.\\
\end{array}
\end{equation}
\end{subequations}
For all $z \in \ll q\rr$, we have  $\CP_{\ll q\rr}(z)=0$ (see Theorem~\ref{theoremPropCharPol}), and so we obtain\footnote{
This result is obtained in a different manner in \cite{PogoruiRDagnino2010} and also in \cite{JanOpf2014}, where a different computational procedure for obtaining $A$ and $B$ is proposed; counting the number of arithmetic operations involved, one can conclude (cf. \cite{FalcaoMirandaSeverinoSoares2017} for complexity and stability analysis) that, for $n>3$, the  process given by (\ref{fAB}) involves less computational effort than the method proposed in \cite{JanOpf2014} and \cite{PogoruiRDagnino2010}.}
\begin{equation}\label{PABz}
P(z)=A+B z.
\end{equation}
Hence, we conclude that  
a coquaternion $z \in \ll q \rr $ is a zero of the polynomial $P$ if and only it satisfies
\begin{equation}\label{condZero}
A+Bz=0.
\end{equation}
We now discuss several cases, depending on the values $A$ and $B$.

\vskip .2cm
\noindent{\bf Case 1} \ \ {\em  $B$ non-singular}

In this case, there is only one zero $z_0$ of $P$ in the class $\ll q \rr$,  given by the formula
\begin{equation}\label{expZ0}
z_0=- B^{-1} A=-\frac{\overline{B}A}{\dt(B)}, 
\end{equation}
as we will now show.
Clearly, $z_0$ given by (\ref{expZ0})  is the unique solution of (\ref{condZero}) and so it remains to prove that it belongs to $ \ll q \rr$. 
From (\ref{DivChar}), we have that
\begin{multline*}
{\cal C}_P(x)
= {\cal C}_{Q}(x) \CP_{\ll q\rr}(x) \CP_{\ll q\rr} (x)+ (A+Bx) \overline{Q}(x)  \CP_{\ll q\rr}(x)\\
 + Q(x) (\overline{A}+\overline{B}x)\CP_{\ll q\rr}(x)+\dt(A)+2 \re(A \overline{B})x+\dt(B)x^2.
\end{multline*}
Let $z$ be an element in $\ll q \rr$ with a non-zero vector part (as observed before, such an element always exists).
Recalling that ${\cal C}_P(z)=0$ and $\CP_{\ll q\rr} (z)=0$, we have
\begin{equation*}
\dt(B)z^2+2 \re(A \overline{B})z+\dt(A)=0 
\end{equation*}
and 
\begin{equation*}
z^2=2 \re(z) z-\dt(z).
\end{equation*}
Hence, we obtain
\begin{equation*}
\dt(B)(2 \re (z) z-\dt(z))+2 \re(A \overline{B})z+\dt(A)=0
\end{equation*}
or
\begin{eqnarray*}
\dt(B)\big(2 \re(z) (\re(z)+\V (z))-\dt(z) \big)+2 \re(A \overline{B})\big(\re(z)+\V (z)\big)+\dt(A)=0
\end{eqnarray*}
which implies
\begin{equation*}
\left\{
\begin{array}{l}
2 \re(z) \big(\dt(B)\re (z) +  \re(A \overline{B})\big)- \dt(B) \dt(z) +\dt(A)=0\\[10pt]
2 \bigl( \dt(B)\re(z) + \re(A \overline{B})\bigr)\V(z)=0
\end{array}\right. .
\end{equation*}
Since $\V(z)\ne 0$, we immediately conclude that 
\begin{equation*}
\dt(B)\re (z) + \re(A \overline{B})=0 \quad {\rm{and}} \quad  \dt(A)- \dt(B) \dt(z)=0 
\end{equation*}
or
\begin{equation*}
\re (z) =- \frac{\re(A \overline{B})}{\dt(B)} \quad {\rm{and}}\quad   \dt(z)=\frac{\dt(A)}{\dt(B)}.
\end{equation*}
Taking into account well-known properties of $\re (q)$ and $\det(q)$, we obtain from the expression (\ref{expZ0}) of $z_0$
\begin{equation*}
\re (z_0)=- \frac{\re(A \overline{B})}{\dt(B)}=\re (z)
\end{equation*}
and
\begin{equation*}
\dt(z_0)=\left(\frac{1}{\dt(B)}\right)^2 \dt(\overline{B})\dt(A)=\frac{\dt(A)}{\dt(B)}=\dt(z)
\end{equation*}
which shows that $z_0\in\ll z \rr=\ll q \rr$.\\[-2ex]

\noindent{\bf Case 2} -  $B=0$\\[-2ex]

\noindent \quad {\bf Case 2.1} - $A\ne 0$

In this case, equation (\ref{condZero}) cannot be satisfied, so there are no zeros  of $P$ in $\ll q \rr$.\\[-2ex]

\noindent \quad {\bf Case 2.2} - $A=0$

From (\ref{condZero}) we have that $P(z)=0 $ for all $z \in \llbracket q \rrbracket$, i.e. the whole hyperboloid $\ll q\rr$ is made up of zeros of $P$.\\[-2ex]

\noindent{\bf Case 3} \ \ {\em  $B$ singular, $B \ne 0$}

{\ }
The roots of $P$ belonging to $\llbracket q \rrbracket$ can be obtained  by studying the solvability of the system
\begin{equation}\label{system}
M_B z =- A
\end{equation}
where $z=(z_0,z_1,z_2,z_3)^{\rm T}$
and  $M_B$ is the matrix defined by (\ref{matrixLq}) (and $A$ is  seen as a  column vector) and selecting the solutions for which  
$\re(z)=q_0$ and $\dv(z)=\dv(q)$.\\[-2ex]

\noindent \quad {\bf Case 3.1} \ \ {\em   $M_B z =- A$ has no solution}

In this case, naturally, we conclude that there are no zeros of $P$ in $\ll q\rr$.\\[-2ex]

\noindent \quad {\bf Case 3.2}   \ \ {\em  $M_B z =- A$ is solvable}

Let  $\delta=(\delta_0,\delta_1,\delta_2,\delta_3)^{\rm T}$ be a solution 
of the system (\ref{system}), i.e. let $M_B \delta=-A$.
As it is well known, the general solution of the system is the sum of a particular solution 
with
the general solution of the associated homogeneous system, i.e. it is
given by 
$
z=u+\delta, 
$
with $u=(u_0,u_1,u_2,u_3)^{\rm T}$ such that  $M_B u=0$.  It can easily be shown that, when $B=b_0+b_1\qi+b_2\qj+b_3\qk $ is singular (but $B\ne 0$), the matrix $M_B$ has rank two and the general solution
$u$ of the homogeneous system $M_B u=0$ is given by
\begin{subequations}\label{CsDs}
\begin{equation}
\left\{
\begin{array}{l}
u_0=\alpha,\\
  u_1=\beta,\\
u_2=k_1 \alpha +k_2\beta,\\
u_3=k_2 \alpha-k_1 \beta; \  \alpha, \beta \in \R,
\end{array}
\right.
\end{equation}
with
\begin{equation}
k_{1}=-\left( \frac{b_0b_2+b_1b_3}{b_0^2+b_1^2}\right)   \quad {\rm and} \quad  k_2=\left( \frac{b_1b_2-b_0b_3}{b_0^2+b_1^2}\right).
\end{equation}
\end{subequations}
So, the solutions of (\ref{condZero}) are $z=z_0+z_1\qi+z_2\qj +z_3\qk $ with
\begin{equation}\label{genSol}
\left\{
\begin{array}{l}
z_0=\alpha+\delta_0,\\ 
z_1=\beta+\delta_1,\\
z_2=k_1 \alpha +k_2\beta+\delta_2,\\
z_3=k_2 \alpha-k_1 \beta+\delta_3; \  \alpha, \beta \in \R.
\end{array}
\right.
\end{equation}
There is no loss in generality in considering  a particular solution $\gamma$  of the system $M_Bz=-A$ with the form 
 $\gamma=(\gamma_0,\gamma_1,0,0)^{\rm T}$. 
This follows immediately by taking
$\alpha=-k_1 \delta_2-k_2 \delta_3$ and $\beta=-k_2 \delta_2+k_1\delta_3$ in (\ref{genSol}) and observing that the values $k_1$ and $k_2$ given by (\ref{CsDs}) satisfy $k_1^2+k_2^2=1$.
So, we can state that the solutions of (\ref{condZero}) are given by 
\begin{equation}\label{gensol12}
\left\{
\begin{array}{l}
z_0=\alpha+\gamma_0,\\ 
z_1=\beta+\gamma_1,\\
z_2=k_1 \alpha +k_2\beta,\\
z_3=k_2 \alpha-k_1 \beta; \  \alpha, \beta \in \R.
\end{array}
\right.
\end{equation}

We now need to select solutions such that $\re(z)=q_0$ and $ \dv(z) =\dv(q)$.
The first condition can always be satisfied provided we take  $\alpha =q_0-\gamma_0$.
On the other hand, 
the expression for the determinant of the vector part of $z$ 
is given  by  
\begin{equation*}
\dv(z)
= (\beta+\gamma_1)^2-(k_1 \alpha +k_2\beta)^2-(k_2 \alpha-k_1 \beta)^2=
-\alpha^2+ \gamma_1^2 +2 \gamma_1 \beta 
\end{equation*}
which, with the choice $\alpha=(q_0-\gamma_0)$, becomes
\begin{equation}\label{dvz}
\dv(z)= -(q_0-\gamma_0)^2+\gamma_1^2+2 \gamma_1 \beta.
\end{equation}

\noindent \quad {\bf Case 3.2 (a)} \ \  {\em There exists $\gamma \in \R$ such that $A+B\gamma =0$}
{\ }

In the special case where $\gamma_1=0$, i.e. when $\gamma=\gamma_0 \in \R$, the expression for $\dv(z)$ simplifies to 
$-(q_0-\gamma_0)^2$ 
and so 
the condition $\dv(z)=\dv(q)$ simply reads as
\begin{equation}\label{condDetgamaReal}
-(q_0-\gamma_0)^2=\dv(q).
\end{equation}
Thus,  there will be zeros in the class $\ll q \rr$ if and only if this condition is satisfied; in such a case, the zeros will form the set
\begin{equation}\label{L}
\quad {\cal L}=\big\{q_0+\beta \qi+\big(k_2\beta+k_1 (q_0-\gamma_0)\big) \qj +
\big(-k_1\beta+ k_2 (q_0-\gamma_0)\big)\qk: \beta \in \R
 \big\},
\end{equation}
with  $k_1,k_2$  given by (\ref{CsDs}). This set of points (considered as points in $\R^4$) can be seen as a line in the hyperplane $x_0=q_0$: the line through the point 
$\big(0, k_1 (q_0-\gamma_0), k_2(q_0-\gamma_0)\big)$ with the direction of the vector $(1,k_2,-k_1)$.
Note also that this is an infinite set, but  a  strict subset of $\ll q\rr$.\\[-2ex]

\noindent \quad {\bf Case 3.2 (b)} \ \ {\em There is no $\gamma \in \R$ such that $A+B\gamma =0$}
{\ }

When $\gamma_1 \ne 0$,  the expression (\ref{dvz}) for the determinant of the vector part of $z$ shows that there is a unique value of $\beta$ for which 
$\dv(z)=\dv(q)$:
\begin{equation}\label{betaP}
\beta=\frac{\dv(q)+(q_0-\gamma_0)^2-\gamma_1^2}{2 \gamma_1}.
\end{equation}
Hence, in this case, we have the following unique zero in the class $\ll q\rr $:
\begin{equation}\label{isolatedZero2}
z_0= q_0+(\beta+\gamma_1)\qi \nonumber 
  + \big(k_2\beta+k_1 (q_0-\gamma_0)\big) \qj+\big(-k_1\beta+ k_2 (q_0-\gamma_0)\big)\qk,\nonumber
\end{equation}
with $\beta$ given by (\ref{betaP}).

\vskip 0.5cm
The previous discussion --- see also  \cite{falcaoICCSA2017} and \cite{JanOpf2017} --- shows that
coquaternionic polynomials may have three different types of zeros, motivating us to introduce the following definition.

\begin{definition}
Let $z$ be a zero of a given coquaternionic polynomial $P$.
\begin{enumerate}
\item $z$ is said to be an isolated  zero of $P$, if $\ll z \rr$ contains no other zeros of $P$;
\item  $z$ is said to be an hyperboloidal zero of $P$,  if $\ll z \rr  \subseteq Z(P)$;
\item $z$ is said to be a linear zero of $P$,  if $z$ is neither an isolated zero  nor an  hyperboloidal zero of $P$.
\end{enumerate}
\end{definition} 
\begin{remark}{\rm 
As far as we are aware, the first authors to note the appearance of zeros which are neither isolated nor hyperboloidal were Janovsk\'a  and Opfer, in \cite{JanOpf2017}, giving an example of a special family of quadratic polynomials with such zeros, which they called {\em unexpected zeros}.
However, we have to point out that there are linear zeros which are not unexpected zeros in the sense of
	\cite{JanOpf2017}, as the examples in Section~4 show.
	}
\end{remark}

In what follows, we will treat all the zeros belonging to the same quasi-similarity class as forming a single zero, i.e., 
we will refer to a whole hyperboloid of zeros or a line of zeros simply as an hyperboloidal zero or a linear zero, respectively.

We may now summarize the results of our previous discussion in the following theorem.
\begin{theorem}\label{thmCharZeros}
Let $\ll q\rr$ be an admissible class of a given polynomial $P\in \HCO[x]$ and let 
$A+Bx$ be the remainder of the right division of $P(x)$ by the characteristic polynomial 
of $\ll q\rr.$ Also, denote by $Z_{\ll q\rr}$ the set of the zeros of $P$ belonging to $\ll q\rr$.
The set $Z_{\ll q\rr}$ can be completely characterized in terms of $A$ and $B$, as follows:

\noindent {\rm 1}.\ If $B$ is non-singular, then $P$ has an isolated zero in $\ll q\rr$ and 
$$
Z_{\ll q\rr}=\Big\{-\frac{\overline{B} A}{\det (B)}\Big\}.
$$
\noindent{\rm 2}.\ 
If $B=0$ and 
\begin{enumerate}
\item[{\rm (a)}] $A \ne 0$, then $Z_{\ll q\rr}=\emptyset$;\\[-2ex]
\item[{\rm (b)}]
 $A=0$, then $Z_{\ll q\rr}=\ll q\rr$, i.e. $\ll q\rr$  is an hyperboloidal zero of $P$.
 \end{enumerate}
 \vskip .12cm
\noindent{\rm 3}.\   If $B \ne 0 $ is singular and the equation $A+Bx=0$ has
\begin{enumerate}
\item[{\rm (a)}]
no solution, then  $Z_{\ll q\rr}=\emptyset$;\\[-2ex]
\item[{\rm (b)}]
a real solution $\gamma_0$, then: 
 
{\rm (i)}
if  $(q_0-\gamma_0)^2=-\dv(q)$, then $$\quad Z_{\ll q\rr}=\big\{q_0+\beta \qi+\left(k_2\beta+k_1 (q_0-\gamma_0)\right) \qj +
\left(-k_1\beta+ k_2 (q_0-\gamma_0)\right)\qk: \beta \in \R
\big\}, $$
\quad with  $k_{1}=-\big( \frac{b_0b_2+b_1b_3}{b_0^2+b_1^2}\big)$  and $k_2=\big( \frac{b_1b_2-b_0b_3}{b_0^2+b_1^2}\big),$
i.e. $Z_{\ll q\rr}$ is a linear zero of $P$;\\

{\rm (ii)}
if $(q_0-\gamma_0)^2\ne -\dv(q)$, then $Z_{\ll q\rr}=\emptyset$;\\[-2ex]

\item[{\rm (c)}]
a nonreal solution $\gamma=\gamma_0+\gamma_1 \qi$, then $P$ has 
an isolated zero in $\ll q\rr$ and 
$$
Z_{\ll q\rr}=\big\{q_0+(\beta+\gamma_1)\qi 
  + \left(k_2\beta+k_1 (q_0-\gamma_0)\right) \qj+\left(-k_1\beta+ k_2 (q_0-\gamma_0)\right)\qk\big\},
$$
with 
$
\beta=\frac{\dv(q)+(q_0-\gamma_0)^2-\gamma_1^2}{2 \gamma_1}\, .
$
\end{enumerate}

\end{theorem}
 
\begin{remark}{\rm 
We should observe that the discussion contained in \cite{JanOpf2017} for the case $B$ singular, $B\ne 0$ is incomplete. In fact, there are two problems associated with the results in that paper.
The first has to do with the fact that the  authors do not consider case  3\,(c) in the previous theorem, assuming implicitly that, when $B$ is singular and $B\ne 0$, roots may only appear if there exists $\gamma \in \R$ such that $A+B\gamma =0$ --- see the assumption in \cite[Theorem 3.2]{JanOpf2017} which is invoked in both Theorems 4.2 and 4.3 of \cite{JanOpf2017}  as a process for computing the roots. 
Second, even when the hypotheses of Theorem 3.2 in \cite{JanOpf2017} are verified (i.e. we are in case 3\,(b) of  Theorem~\ref{thmCharZeros}), the formula 
$$z_0=\alpha\overline{B}+\gamma, \ \  \alpha \in \R,$$
which is  proposed in \cite{JanOpf2017} for obtaining the solutions  
of equation $A+Bz=0$ may not give  all the solution of this equation; as we noted, the general solution of
$A+Bz$ is given by  $z=u+\gamma $, with $u$ any coquaternion satisfying $B u=0$,   $u$ not necessarily  of the special form $\alpha \overline{B}$. 
  
For the above reasons, 
the use of the algorithm proposed by the authors of \cite{JanOpf2017} to compute the roots of a given polynomial --- Algorithm 6.1 --- may fail to produce all the roots.
Examples~\ref{example1}--\ref{example2}  given later illustrate our observations. 
A revised version of the Algorithm 6.1 of \cite{JanOpf2017} to compute all the roots of a given polynomial  will be presented in the end of this section.
}
  \end{remark}

We now give a theorem
with a special case where we know that a  polynomial  has linear zeros.

\begin{theorem}\label{theorLinZeros}
Let $P(x)$ be a polynomial of degree $n$ whose companion polynomial has $m$ real simple zeros
$r_1, r_2, \ldots, r_{m}, \ m\le 2n,$  and let $P_{r}(x)=P(x) (x-r)$ with $r \in \R$,  $r \ne r_i; i=1,\ldots, m$.
Then, $P_r(x)$ has (at least) $m$ linear zeros.
\end{theorem}
\begin{proof}
Naturally, the  roots of the companion polynomial of $P_r$ are the previous roots of 
the companion polynomial of $P$ together with the double root $r$. So 
$r,\, r,r_1, r_2, \ldots, r_{m}$ are roots of ${\cal C}_{P_r}$.
This means that
the classes $\ll p_k\rr$, with $p_k=\frac{r+r_k}{2}+\frac{r-r_k}{2} \qj; k=1,\ldots,m,$ are admissible 
classes for $P_r$. We now show that each of these classes contains a linear zero.
The characteristic polynomial of the class $\ll p_k\rr $ is $\CP_{\ll p_k\rr}(x)=(x-r)(x-r_k)$. Dividing 
$P_r(x)$ by $\psi_{\ll p_k\rr}$ will give us
$$P_r(x)=Q(x) (x-r_k)(x-r)+Bx+A.$$
But, since $P_r(r)=0$, we obtain $Br +A=0$.
We cannot have $B$ nonsingular, since this would imply  $z_0=r \in \ll p_k\rr$, which is false, since  $\re(p_k) \ne r$.
Also, we cannot have $A=B=0$: this would mean that $P_r(x)=Q(x)(x-r) (x-r_k)$ would have the real root $r_k$, implying that 
${\cal C}_{P_r}$ would have $r_k$ as a double root, which is contrary to the hypotheses of the theorem.
 Hence, we are in the case 3\,(b) of Theorem~\ref{thmCharZeros}, with $\gamma_0=r$; since
$$(\re(p_k)-r)^2=\left(\textstyle{\frac{r+r_k}{2}}-r\right)^2=\left(\textstyle{\frac{r-r_k}{2}}\right)^2=-\dv(p_k),$$
we may conclude that $\ll p_k\rr$ contains a line of zeros.
\qquad
\end{proof}

The following observations clarify some characteristics  of the zero-sets of coquaternionic polynomials.
\begin{enumerate}
\item[O1.]
We first note that classes of the type $\ll q_0+\sqrt{\dv(q)}\,  \qi\rr,$ $\dv(q)>0$, will never contain a linear zero.
In fact, as shown in Theorem~\ref{thmCharZeros}, linear zeros are only obtained in case 3\,(b)-(i), demanding   $(q_0-\gamma_0)^2=-\dv(q)$ to be satisfied; this condition is, naturally, impossible to verify if $\dv(q)>0$. 
(The same conclusion could be taken invoking geometric arguments,  by observing that an hyperboloid of two-sheets does not contain any straight line.)
So, we conclude that  linear zeros will always be either of Type 2 or Type 3 coquaternions.
\item[O2.]
We have seen that a  polynomial $P$ of degree $n$ can attain the maximum number $n(2n-1)$ of zeros only if the roots of its companion polynomial are all real and simple. However, we must point out that  this is not a sufficient condition  for the existence of $n(2n-1)$ zeros, as  illustrated in  Example~\ref{example3}. 

\item[O3.] A coquaternion $z$ is an hyperboloidal zero of a given polynomial $P$ if and only if $\CP_{\ll z\rr}$ divides $P$. This implies that $(\CP_{\ll z\rr})^2$ has to divide the companion polynomial, meaning that the roots in $\C$ of
$\CP_{\ll z\rr}$ will  appear as roots of ${\cal C}_P$ with a multiplicity greater or equal to two if 
$z$ is of Type 1 or Type 2 or greater or equal to 4 if $z$ is of Type 3.  However, not all  multiple roots of ${\cal C}_P$ correspond to hyperboloidal zeros of $P$; see Example~\ref{example4}.

\end{enumerate}

We end this section by giving an algorithm to compute the roots of a coquaternionic polynomial which is based on the discussion 
given previously.

\begin{algorithm}[h!]
\caption{Compute the roots of a coquaternionic polynomial $P$}\label{algor}

{\sc Input:} Coefficients of $P$\\[-10pt]

{\sc Output:} Lists $\ell_I$,$\ell_H$ and $\ell_L$ with the isolated, hyperboloidal and linear roots of $P$ \\[-10pt]

\begin{algorithmic}[1]
\State Initialize lists ${\ell}_I$,  $\ell_H$  and  ${\ell}_L$ as empty lists
\State Compute  ${\cal C}_P$ -   formula (\ref{compPolynomial}) \Comment{${\cal C}_P$  companion pol. of $P$}
\State Determine roots in $\C$ of ${\cal C}_P$ \Comment{Use a numerical method, if necessary}
\State Identify admissible classes of $P$ - formulas (\ref{chPcomplex}) and (\ref{chPreal})
\For{each admissible class} 
\State Compute $A$ and $B$ - formulas (\ref{fAB})
\If {$B=0$}
 \If {$A=0$} 
 \State {Add representative of class to list ${\ell_H}$} 
 \EndIf
\Else 
\State {Compute $\dt(B)$ - formula (\ref{detq})}
\If {$\dt(B)\ne 0$} 
\State Compute $z_0$ by formula (\ref{expZ0}) and add it to list $\ell_I$
\Else 
\If{system $Bx=-A$ has a solution $\gamma$}
\If{$\gamma \in \R$}
\If {condition (\ref{condDetgamaReal}) holds}
\State Determine $k_1,k_2$ by formulas (\ref{CsDs})
\State  Add  
 $\{q_0,k_1,k_2\}$ to list $\ell_L$ \Comment{$q_0$ real part of rep. of class}
 \EndIf
\Else
\State Determine solution $z_0$ by formula (\ref{isolatedZero2})
\State Add $z_0$  to list $\ell_I$ 
\EndIf
\EndIf
\EndIf
\EndIf
\EndFor
\end{algorithmic}
\end{algorithm}

\section{Examples}\label{section4}
In this section, we present several examples of application of Algorithm~\ref{algor} which illustrate some of the results and remarks 
given in Section~\ref{section3}.

 \begin{example}\label{example1}
 {\rm
Let  $P_1(x)=x^2+(1+\qi+\qj+\qk)x+(2-\qi-\qj+\qk)$. 
The roots of the companion polynomial of $P_1$ 
are $\qi, -\qi, -1- \sqrt{2}\qi , -1+\sqrt{2} \qi$
and so we have the following two admissible 
 classes: $\ll \qi\rr , \ll  -1+\sqrt{2} \qi\rr$.

Consider the determination of the roots lying in class $\ll\qi\rr$.
Dividing  the polynomial $P_1$ by
 $\CP_{\ll\qi\rr} $ leads to
\begin{eqnarray*}
A=1-\qi-\qj +\qk \quad {\rm and} \quad  B=1+\qi+\qj+\qk.
\end{eqnarray*}
So, we  have that  $B$ is singular, but 
there is no $\gamma \in \R$ such that $A+B\gamma =0$.
However, it can be easily verified that the system $M_Bz=-A$ is consistent and so we are in case 3\,(c) of Theorem~\ref{thmCharZeros}, obtaining that $z_0=\qi$ is the only root of $P_1$ in $\ll \qi\rr$.  This shows that 
the  root $z=\qi$ will not  be computed if we use Algorithm 6.1 in \cite{JanOpf2017}.
}
\end{example}

 \begin{example}\label{example2}
 {\rm 
We now consider the polynomial  $P_2(x)=x^2-(3+\qj)x+2+\qj$, whose corresponding companion polynomial has a triple root $1$ and a simple root $3$. Hence, $P_2$ has two admissible classes: $\ll 1\rr$ and $\ll 2+\qj\rr$. If we divide $P_2$ by the characteristic polynomial of the class
$\ll 1\rr$ we obtain 
$$A=1+\qj\quad {\rm and} \quad  B=-1-\qj.$$
Hence, there exists $\gamma=1 \in \R$  such that $A+B\gamma=0$; also, it is easily seen that the  condition (\ref{condDetgamaReal}) is verified and so we are in the presence of a linear zero. More precisely, we conclude that all the elements in the set
 $\{1+\beta\qi+\beta\qk : \beta\in \R\}$
 are zeros of $P_2$.
 
In a total analogous manner we find that $P_2$ has a linear zero in the class $\ll 2+\qj\rr$,
given by  $
\left\{2+\beta\qi +\qj -\beta\qk  : \beta \in \R\right \}
$. The application of  \cite[Algorithm~6.1]{JanOpf2017} would 
only determine two isolated zeros, namely $z_0=1$ and $z_0=2+\qj$.
}
\end{example}
\vskip .3cm

\begin{example}\label{example3New}
{\rm 
We consider now the polynomial    $P_3(x)=P(x)(x-1)$, where 
$P(x)=a_3x^3+a_2 x^2+a_1 x+a_0$ with 
$a_3=2+2\qi-\qj$, $a_2=-1-5 \qj-\qk$, $a_1=-4-5 \qi +\qj +\qk$ and $a_0=2-2\qi+2\qj+3 \qk$  which was studied in \cite[Example~7.2]{JanOpf2017}. The companion polynomial of $P$  has 6 simple real roots, as  observed in that paper; hence, by applying Theorem~\ref{theorLinZeros}, we know that $P_3$ has  6 linear roots which, however, are not found in \cite{JanOpf2017}.
As an example of a line of zeros of $P_3$, we have
$$
{\cal L}=\bigl\{\beta \qi+(-\textstyle{\frac{3}{5}}\beta-\textstyle{\frac{4}{5}}) \qj + (-\textstyle{\frac{4}{5}}\beta+\textstyle{\frac{3}{5}})\qk : \beta \in \R\big\} \subsetneqq  \ll \qj\rr,$$
as can be easily verified.
}
\end{example}

\begin{example}\label{example3}{\rm 
Let $P_4(x)=x^2+(-5-\qj)x+(\frac{11}{2}+\frac{5}{2}\qj)$. This second degree polynomial
has a companion polynomial with four simple real roots  $1,2,3,4$ and hence has 6 admissible classes, which is the maximum number allowed for a second degree polynomial. However, as we will now show, not all the admissible classes contain roots of $P_4$.
If we consider, for example, the class $\ll 2+\qj\rr$, we obtain
$$A=\textstyle{\frac{5}{2}}+\textstyle{\frac{5}{2}}\qj \quad {\rm and} \quad  B=-1-\qj.$$
Hence, $B$ is singular, $B\ne 0$ and there exists $\gamma=\frac{5}{2} \in \R$ such that $A+B\gamma =0$; in this case, we have
$(\frac{5}{2}-2)^2=\frac{1}{4} \ne 1=-\dv(2+\qj)$, and so 
we are in case 3\,(b)-(ii) of Theorem~\ref{thmCharZeros}, leading us to conclude that there are no roots in this class. In the same manner, we show that the class $\ll 3+\qj\rr$ does not contain any root of $P_4$. This illustrates our observation O2 made in Section~\ref{section3}.
}
\end{example}

\begin{example}\label{example4}
{\rm We now consider three very simple polynomials that illustrate very clearly the observation O3 made in Section~\ref{section3}:
\begin{eqnarray*}
P_5(x)&=& x^2-2x+1,\\[.1cm]
Q_5(x)&=&x^2-(2+\qi+\qj)x+(1+\qi+\qj),\\[.1cm]
R_5(x)&=&
x^2 +(-2-6\qi-5\qj-3\qk)x+(3\qi+2\qj+2\qk).
\end{eqnarray*}
All these polynomials  have $(x-1)^4$ as companion polynomial (i.e., the companion polynomial has the real root 1 with multiplicity 4). However, 
in what concerns the zeros in the (unique) class $\ll 1\rr$, they behave differently:   the polynomial $P_5$ has 
$\ll 1\rr$ as an hyperboloidal zero, the polynomial 
$Q_5$ has the linear zero $\{1+\beta \qi+\beta \qj : \beta \in \R\}$ and the polynomial $R_5$ has the isolated zero 
$z_0=1+5\qi+4\qj+3\qk$. 

Note also that, in the case of $R_5$, the real multiple  root $1$ of the companion polynomial
is not a real root of $R_5$, contradicting the last assertion in Theorem~4.3 of \cite{JanOpf2017}.
}
\end{example}

Our last example addresses the problem posed Janovsk\'a and Opfer in \cite{JanOpf2017}:
``{\textit{Given $n>4$, can we find a coquaternionic polynomial of degree $n$ with the maximal number ${2n \choose 2}$ of zeros?}}".

\begin{example}\label{example5}{\rm 
Let  $P_6(x)=x^5+c_4x^4+c_3 x^3+c_2x^2+c_1x+c_0$ with
$$
\begin{array}{ll}
 c_4=\frac{1}{2}+\qi+7\qj+\frac{13}{2}\qk, \ c_3=24+\frac{13}{2}\qi+\frac{11}{2}\qj-6\qk,\ c_2=-\frac{47}{2}-32\qi-18\qj+\frac{33}{2}\qk, \\[.1cm]
c_1=-51+\frac{81}{2}\qi-\frac{57}{2}\qj-52\qk, \  c_0=-9-12\qi-18\qj+9\qk.
\end{array}
$$
The  companion polynomial of this fifth degree polynomial has 10 real simple roots, $-6,-3,-1,\, 1,\, 2,5, 1\pm\sqrt{2},\frac{1}{2}(-1\pm\sqrt{5})$, and 
it can be verified that each of the 45 admissible classes contains an isolated root of $P_6$.
 For example, in the class $\ll 1+\sqrt{2}\qj\rr$ we have the root
$z_0=1+\frac{3}{8} \qi -\frac{11}{8} \qj -\frac{1}{2} \qk $ and in the class $\ll \qj\rr$ the isolated root $z_1=\frac{17}{84}\qi-\frac{27}{28}\qj-\frac{1}{3}\qk .$ 
Examples of polynomials of degrees 6,\, 7,\, 8 and 9 achieving the maximum possible number of roots --- 66, 91, 120 and 153, respectively ---  
were also constructed; all these polynomials were obtained as products of appropriately chosen linear factors and  we believe that this type of process can be used to compute examples of polynomials of any prescribed degree with the maximum number of roots. }
\end{example}

\subsection*{Acknowledgments}
Research at CMAT was financed by Portuguese Funds through FCT - Funda\c c\~ao para a Ci\^encia e a Tecnologia, within
the Project 
UID/MAT/00013/2013. 
Research at NIPE was carried out within the funding with COMPETE reference number POCI-01-0145-FEDER-006683 (UID/ECO/03182/2013), with the FCT/MEC's (Funda\c c\~ao para a Ci\^encia e a Tecnologia, I.P.) financial support through national funding and by the ERDF through the Operational Programme on ``Competitiveness and Internationalization - COMPETE 2020" under the PT2020 Partnership Agreement.

\bibliographystyle{siam}

\end{document}